\documentclass[reqno]{amsart}

\usepackage{amssymb}
\usepackage{graphicx}
\usepackage{amscd}
\usepackage[hidelinks]{hyperref}
\usepackage{color}
\usepackage{float}
\usepackage{graphics,amsmath,amssymb}
\usepackage{amsthm}
\usepackage{amsfonts}
\usepackage{latexsym}
\usepackage{epsf}
\usepackage{xifthen}
\usepackage{mathrsfs}
\usepackage{dsfont}
\usepackage{makecell}
\usepackage[FIGTOPCAP]{subfigure}
\usepackage{amsmath}
\allowdisplaybreaks[4]
\usepackage{listings}
\usepackage{etoolbox}
\usepackage{fancyhdr}
\usepackage{pdflscape}
\usepackage[title,toc,titletoc]{appendix}
\usepackage{enumitem}
\usepackage[noadjust]{cite}
\usepackage{tikz}
\usepackage{ytableau}

 \newtheoremstyle{mytheorem}% name % cf. thmtest.tex of AMSLaTeX
 {3pt}%      Space above
 {3pt}%      Space below
 {\slshape}% Body font
 {}%         Indent amount (empty = no indent,
 {\bfseries}% Thm head font
 {.}%        Punctuation after thm head
 { }%        Space after thm head: " " = normal interword space;
 {}%         Thm head spec (can be left empty, meaning `normal')

\numberwithin{equation}{section}

\theoremstyle{theorem}
\newtheorem{theorem}{Theorem}[section]
\newtheorem*{theorem*}{Theorem}

\newtheorem{lemma}[theorem]{Lemma}

\newtheorem{problem}{Problem}[section]

\newtheorem{innercustomgeneric}{\customgenericname}
\providecommand{\customgenericname}{}
\newcommand{\newcustomtheorem}[2]{%
	\newenvironment{#1}[1]
	{%
		\renewcommand\customgenericname{#2}%
		\renewcommand\theinnercustomgeneric{##1}%
		\innercustomgeneric
	}
	{\endinnercustomgeneric}
}
\newcustomtheorem{ctheorem}{Theorem}
\newcustomtheorem{clemma}{Lemma}

\theoremstyle{definition}

\newtheorem*{definition*}{Definition}

\newtheorem*{example*}{Example}
\newtheorem*{examples*}{Examples}
\newtheorem{remark}{Remark}[section]
\newtheorem*{remark*}{Remark}
\newtheorem*{remarks*}{Remarks}

\newtheoremstyle{named}{}{}{\itshape}{}{\bfseries}{.}{.5em}{#1\thmnote{ #3}}
\theoremstyle{named}

\newcommand{\Keywords}[1]{\ifthenelse{\isempty{#1}}{}{\smallskip \smallskip \noindent \textbf{Keywords}. #1}}
\newcommand{\MSC}[2][2020]{\ifthenelse{\isempty{#2}}{}{\smallskip \smallskip \noindent \textbf{#1MSC}. #2}}
\newcommand{\abstractnote}[1]{\ifthenelse{\isempty{#1}}{}{\smallskip \smallskip \noindent \textsuperscript{\dag}#1}}

\makeatletter
\def\specialsection{\@startsection{section}{1}%
  \z@{\linespacing\@plus\linespacing}{.5\linespacing}%
  {\normalfont}}% NEW
\def\section{\@startsection{section}{1}%
  \z@{.7\linespacing\@plus\linespacing}{.5\linespacing}%
  {\normalfont\scshape}}% NEW
\patchcmd{\@settitle}{\uppercasenonmath\@title}{\Large\boldmath}{}{}
\patchcmd{\@settitle}{\begin{center}}{\begin{flushleft}}{}{}
\patchcmd{\@settitle}{\end{center}}{\end{flushleft}}{}{}
\patchcmd{\@setauthors}{\MakeUppercase}{\normalsize}{}{}
\patchcmd{\@setauthors}{\centering}{\raggedright}{}{}
\patchcmd{\section}{\scshape}{\large\bfseries\boldmath}{}{}
\patchcmd{\subsection}{\bfseries}{\bfseries\boldmath}{}{}
\renewcommand{\@secnumfont}{\bfseries}
\patchcmd{\@startsection}{\@afterindenttrue}{\@afterindentfalse}{}{}
\patchcmd{\abstract}{\leftmargin3pc}{\leftmargin1pc}{}{}

\def\maketitle{\par
  \@topnum\z@ % this prevents figures from falling at the top of page 1
  \@setcopyright
  \thispagestyle{empty}% this sets first page specifications
  \ifx\@empty\shortauthors \let\shortauthors\shorttitle
  \else \andify\shortauthors
  \fi
  \@maketitle@hook
  \begingroup
  \@maketitle
  \toks@\@xp{\shortauthors}\@temptokena\@xp{\shorttitle}%
  \toks4{\def\\{ \ignorespaces}}% defend against questionable usage
  \edef\@tempa{%
    \@nx\markboth{\the\toks4
      \@nx\MakeUppercase{\the\toks@}}{\the\@temptokena}}%
  \@tempa
  \endgroup
  \c@footnote\z@
  \@cleartopmattertags
}
\makeatother

\newcommand{\bE}{\mathbf{E}}

\newcommand{\cH}{\mathcal{H}}

\newcommand{\sfP}{\mathsf{P}}
\newcommand{\sfN}{\mathsf{N}}

\newcommand{\Prob}{\operatorname{Pr}}
\newcommand{\li}{\operatorname{li}}

\title[Hitting a prime]{Hitting a prime by rolling a die with infinitely many faces}

\author[S. Chern]{Shane Chern}
\address{Department of Mathematics and Statistics, Dalhousie University, Halifax, NS, B3H 4R2, Canada}
\email{chenxiaohang92@gmail.com}

\begin{document}

\maketitle

\begin{abstract}

Alon and Malinovsky recently proved that it takes on average $2.42849\ldots$~rolls of fair six-sided dice until the first time the total sum of all rolls arrives at a prime. Naturally, one may extend the scenario to dice with a different number of faces. In this paper, we prove that the expected stopping round in the game of Alon and Malinovsky is approximately $\log M$ when the number $M$ of die faces is sufficiently large.


\MSC{60C05, 11A41.}
\end{abstract}

\section{Introduction}

Recently, Alon and Malinovsky \cite{AM2023} considered the following game:

\begin{problem}
	Assuming that $X_1,X_2,\ldots$ are independent uniform random variables on the integers $1,2,\ldots,6$, we add these random variables term by term, and stop at round $\tau$ when $X_1+\cdots + X_\tau$ reaches a prime number at the first time. Then what is the value of the expectation $\bE(\tau)$?
\end{problem}

This problem first appeared as a \emph{Student Puzzle} in the Bulletin of the Institute of Mathematical Statistics \cite{Das2017}. To gain a basic understanding of how $\bE(\tau)$ is computed, we note that the game stops at \emph{Round 1} when $(X_1)$ is any of $(2)$, $(3)$ and $(5)$, each of probability $\frac{1}{6}$, at \emph{Round 2} when $(X_1,X_2)$ is any of $(1,1)$, $(1,2)$, $(1,4)$, $(1,6)$, $(4,1)$, $(4,3)$, $(6,1)$ and $(6,5)$, each of probability $\frac{1}{6}\times \frac{1}{6} = \frac{1}{36}$, and so on. So
\begin{align*}
	\bE(\tau) = 1\times \frac{3}{6} + 2\times \frac{8}{36} + \cdots .
\end{align*}
It is notable that this game may not terminate; an example is $(X_1,X_2,X_3,\ldots) = (6,6,6,\ldots)$, a sequence of random variables equal to $6$. This observation indicates that $\bE(\tau)$ is an infinite sum in terms of probabilities. However, it was shown by Alon and Malinovsky \cite{AM2023} that $\bE(\tau)$, surprisingly, has a finite value of around $2.43$.

Although \emph{God} may play dice with the universe, \emph{He} does not play a six-sided die at all times.\footnote{Albert Einstein, in a letter to Max Born on Dec 4\textsuperscript{th} 1926, proclaimed that ``\emph{God} does not play dice with the universe.'' See \cite[pp.~90--91, Letter 52]{BE1971}.} It is then natural to ask how $\bE(\tau)$ behaves when the number of faces of the fair die varies. Consideration along this line is not new, and in fact, Martinez and Zeilberger \cite{MZ2023} already algorithmically analyzed cases of playing a die with up to $40$ faces. But once again, the expected dice rolls in \cite{MZ2023} were still estimated numerically, just akin to what was done by Alon and Malinovsky \cite{AM2023}. Thus, it remains meaningful to look for a concrete pattern between the number of dice faces and the corresponding expectation of stopping round.

Of course, the title of this paper is somewhat misleading, as one cannot create a fair die with infinitely many faces. So what really interests us is the case where the number of die faces is sufficiently large. As shown by Alon and Malinovsky \cite{AM2023}, to tackle this \emph{hitting-a-prime puzzle}, a crucial ingredient concerns the enumeration of primes in intervals of length $M$, the number of die faces. Such enumerations usually behave chaotically when we encounter \emph{short} intervals, and this fact explains why the expected stopping round $\bE(\tau)$ is only numerically evaluated in \cite{AM2023} and \cite{MZ2023}. However, when the interval length gets larger, the scenario becomes more predictable in the asymptotic sense.

To formalize our setting, we assume that $\eta_1,\eta_2,\ldots$ are independent uniform random variables on the integers $1,2,\ldots,M$. These random variables are added term by term, and the process stops at round $\tau^{(M)}$, which is called the \emph{stopping round}, when $\eta_1+\cdots + \eta_{\tau^{(M)}}$ reaches a prime number at the first time. In addition, we define
\begin{align*}
	\cH:=\left\{(\eta_1,\ldots,\eta_r)\in \{1,2,\ldots,M\}^r:\text{$\textstyle\sum\limits_{i=1}^\ell \eta_i \not\in \mathbb{P}$ for all $1\le \ell < r$}\right\}.
\end{align*}
Here $\mathbb{P}$ is the set of prime numbers. For $\eta=(\eta_1,\ldots,\eta_r)\in \cH$, we denote its \emph{size} by $|\eta|:=\sum_{i=1}^r \eta_i$, and its \emph{number of rounds} by $r(\eta):=r$. Meanwhile, we may partition $\cH$ into two disjoint subsets $\cH=\cH_\sfP \sqcup \cH_\sfN$, where
\begin{align*}
	\cH_\sfP&:=\{\eta\in \cH: |\eta|\in \mathbb{P}\},\\
	\cH_\sfN&:=\{\eta\in \cH: |\eta|\not\in \mathbb{P}\}.
\end{align*}
Then the probability $\Prob_{\sfP}(r)$ of having stopping round $r$ is
\begin{align*}
	\Prob_{\sfP}(r):= \big(\tfrac{1}{M}\big)^r\cdot \#\{\eta\in \cH_\sfP:r(\eta) = r\}.
\end{align*}
Our objective is to analyze the expected stopping round
\begin{align}\label{eq:E-def}
	\bE(\tau^{(M)}):=\sum_{r\ge 1} r \Prob_{\sfP}(r).
\end{align}

We start by applying the algorithm of Martinez and Zeilberger \cite{MZ2023} to approximate $\bE(\tau^{(M)})$ for different choices of $M$, with the summation in \eqref{eq:E-def} truncated by $1\le r\le 50$. From the data in Table \ref{ta:value}, we are driven to an asymptotic relation as stated in the main theorem.

\begin{table}[ht]
	\caption{Approximated values of $\bE(\tau^{(M)})$}\label{ta:value}
	\renewcommand\arraystretch{1.25}
	\noindent
	\begin{tabular}{ccc}
		\hline
		$M$ & $\bE(\tau^{(M)})$ & $\log M$\\
		\hline
		$10$ & $\approx 2.97$ & $2.30258\ldots$\\
		$20$ & $\approx 3.47$ & $2.99573\ldots$\\
		$50$ & $\approx 4.51$ & $3.91202\ldots$\\
		$100$ & $\approx 5.27$ & $4.60517\ldots$\\
		$200$ & $\approx 5.91$ & $5.29831\ldots$\\
		$500$ & $\approx 6.89$ & $6.21460\ldots$\\
		$1000$ & $\approx 7.59$ & $6.90775\ldots$\\
		\hline
	\end{tabular}
\end{table}

\begin{theorem}\label{th:main}
	As $M\to +\infty$,
	\begin{align}\label{eq:main}
		\bE(\tau^{(M)}) = \log M + O(\log\log M).
	\end{align}
\end{theorem}

\subsubsection*{Notation}

We adopt the conventional Bachmann--Landau symbols. The big-$O$ notation is defined in the sense that $f(x)=O\big(g(x)\big)$ means that there exists an absolute constant $C$ such that $|f(x)|\le C g(x)$. Also, we have the small-$o$ notation: $f(x)=o\big(g(x)\big)$ means that $\lim f(x)/g(x)=0$. Finally, if $\lim f(x)/g(x)=1$, then we write $f(x)\sim g(x)$.

\section{Counting primes}

\subsection{The prime-counting function}

We begin with a collection of known results on the \emph{prime-counting function} $\pi(x)$, which enumerates the number of primes less than or equal $x$, where $x$ is a real number. The growth rate of the prime-counting function is depicted by the \emph{Prime Number Theorem}:
\begin{align}
	\pi(x)\sim \frac{x}{\log x}\sim \li(x),
\end{align}
as $x\to +\infty$, where $\li(x)$ is the \emph{logarithmic integral function}
\begin{align*}
	\li(x):=\int_0^x \frac{dt}{\log t}.
\end{align*}
More precisely, Rosser and Schoenfeld \cite[p.~69, eqs.~(3.4) \& (3.5)]{RS1962} established explicit upper and lower bounds for $\pi(x)$:
\begin{alignat}{2}
	\pi(x)&< \frac{x}{\log x - 1.5} && \qquad\qquad(x\ge 5),\label{eq:pi-upper}\\
	\pi(x)&> \frac{x}{\log x} && \qquad\qquad(x\ge 17).\label{eq:pi-lower}
\end{alignat}
The approximation of $\pi(x)$ by $\li(x)$ is sharper: De la Vall\'ee Poussin \cite{dlV1899} proved that for a certain positive constant $a$,
\begin{align*}
	\pi(x)=\li(x)+O\big(xe^{-a\sqrt{\log x}}\big).
\end{align*}
See also \cite[p.~65, Theorem 23, eq.~(43)]{Ing1990}. Explicit values of the constant $a$ were computed by, for example, Ford \cite{For2004}. Since $e^{a\sqrt{\log x}}$ increases more rapidly than any positive power of $\log x$, we may loosen the above error term as
\begin{align*}
	\pi(x)=\li(x)+o\left(\frac{x}{(\log x)^5}\right).
\end{align*}
In other words, there is a positive constant $C$ such that for all $x\ge C$,
\begin{align}\label{eq:pi-li}
	|\pi(x)-\li(x)|<\frac{x}{2(\log x)^5}.
\end{align}

\subsection{Assumptions and parameters}

Recall that the purpose of this paper is to study the asymptotic behavior of $\bE(\eta^{(M)})$ as $M\to +\infty$. Thus, it is safe to look at sufficiently large $M$. That is to say, we may choose a positive constant $M_0>e$ whose exact value is of \emph{no} concern, and always assume that $M\ge M_0$. Here $M_0$ satisfies a couple of conditions as follows:
\begin{itemize}
	\item[\textbf{(C.1)}] For all $x\ge \frac{M_0}{(\log M_0)^3}$, the inequalities \eqref{eq:pi-upper}, \eqref{eq:pi-lower} and \eqref{eq:pi-li} hold.
\end{itemize}

To simplify our analysis of the error terms, it is helpful to further assume that $M_0$ is such that
\begin{itemize}
	\item[\textbf{(C.2)}] The function $\frac{x}{(\log x)^5}$ is increasing on the interval $\big[\frac{M_0}{(\log M_0)^3},+\infty\big)$.
\end{itemize}

\begin{remark}
	Since $\log x$ is positive-valued and increasing on $(1,+\infty)$, it is clear that with the above assumption, the functions $\frac{x}{(\log x)^\delta}$ with $0<\delta\le 5$ are also increasing on the same interval.
\end{remark}

Meanwhile, in our derivation of Lemma \ref{le:pi-interval}, there will be a handful of inequalities (exclusively in $M$) that are marked by ``~$!$~'', so $M_0$ is chosen such that
\begin{itemize}
	\item[\textbf{(C.3)}] For all $M\ge M_0$, these inequalities are valid.
\end{itemize}

\begin{remark}
	For example, we will need to use
	\begin{align*}
		\frac{M}{\log M}-\frac{M}{(\log M)^3} \ge_! \frac{M}{\log M+4\log\log M}.
	\end{align*}
	Clearly, this is true when $M$ is large enough since
	\begin{align*}
		\left(\frac{M}{\log M}-\frac{M}{(\log M)^3}\right) - \frac{M}{\log M+4\log\log M} \to +\infty,
	\end{align*}
	as $M\to +\infty$, thereby explaining why $M_0$ could be chosen. We leave the (\emph{easy}!) examination of other such inequalities to the interested reader.
\end{remark}

Furthermore, to make the statement in \eqref{eq:pi-interval-2} robust, we need to require that $M_0$ is such that
\begin{itemize}
	\item[\textbf{(C.4)}] For all $M\ge M_0$, we have $\log M-4\log\log M>1$.
\end{itemize}

Finally, we choose the following two parameters that are critical in our analysis:
\begin{align}
	R_1 &:= \lfloor(\log M)^3\rfloor,\label{eq:para-1}\\
	R_2 &:= M^2.\label{eq:para-2}
\end{align}
In addition, the last condition imposed on $M_0$ is:
\begin{itemize}
	\item[\textbf{(C.5)}] For all $M\ge M_0$, we have $1<R_1<R_2$.
\end{itemize}

\subsection{Counting primes in intervals of length $M$}

Now our core concern lies in bounding the number of primes in the interval $[k+1,k+M]$ for small $k$.

\begin{lemma}\label{le:pi-interval}
	Suppose $M\ge M_0$. Then for any integer $k$ with $0\le k\le M (\log M)^3$,
	\begin{align}
		\pi(k+M)-\pi(k)&\ge \frac{M}{\log M+4\log\log M},\label{eq:pi-interval-1}\\
		\pi(k+M)-\pi(k)&\le \frac{M}{\log M-4\log\log M}.\label{eq:pi-interval-2}
	\end{align}
\end{lemma}

\begin{proof}
	We first consider the case where $0\le k\le \frac{M}{(\log M)^3}$. For the lower bound, it is clear that
	\begin{align*}
		\pi(k+M)-\pi(k) &\ge \pi(k+M)-k\\
		&\ge \pi(M)-\frac{M}{(\log M)^3}\\
		\text{\tiny (by \eqref{eq:pi-lower})}&\ge \frac{M}{\log M}-\frac{M}{(\log M)^3}\\
		&\ge_! \frac{M}{\log M+4\log\log M}.
	\end{align*}
	For the upper bound, we have
	\begin{align*}
		\pi(k+M)-\pi(k) &\le \pi(k+M)\\
		&\le \pi\big(\tfrac{M}{(\log M)^3}+M\big)\\
		\text{\tiny (by \eqref{eq:pi-upper})}&\le \frac{\frac{M}{(\log M)^3}+M}{\log \big(\frac{M}{(\log M)^3}+M\big)-1.5}\\
		&\le \frac{\frac{M}{(\log M)^3}+M}{\log M-1.5}\\
		&\le_! \frac{M}{\log M-4\log\log M}.
	\end{align*}

	When $\frac{M}{(\log M)^3}<k\le M(\log M)^3$, we shall estimate $\pi(k+M)-\pi(k)$ by $\li(k+M)-\li(k)$. Let us begin with the error term in light of \eqref{eq:pi-li}:
	\begin{align*}
		&\big|\big(\pi(k+M)-\pi(k)\big)-\big({\li(k+M)}-\li(k)\big)\big|\\
		&\qquad\qquad\le \big|\pi(k+M)-\li(k+M)\big| + \big|\pi(k)-\li(k)\big|\\
		&\qquad\qquad\le \frac{k+M}{2\big(\!\log(k+M)\big)^5} + \frac{k}{2\big(\!\log k\big)^5}\\
		&\qquad\qquad\le \frac{M(\log M)^3+M}{\big(\!\log(M(\log M)^3+M)\big)^5}\\
		&\qquad\qquad\le \frac{M(\log M)^3+M}{(\log M)^5}.
	\end{align*}
	Meanwhile,
	\begin{align*}
		\li(k+M)-\li(k) = \int_k^{k+M}\frac{dt}{\log t}.
	\end{align*}
	It follows that
	\begin{align*}
		\li(k+M)-\li(k) &\ge \frac{M}{\log (k+M)}\\
		&\ge \frac{M}{\log (M(\log M)^3+M)},
	\end{align*}
	and that
	\begin{align*}
		\li(k+M)-\li(k) &\le \frac{M}{\log k}\\
		&\le \frac{M}{\log (M(\log M)^{-3})}.
	\end{align*}
	Finally, using the facts that
	\begin{align*}
		\frac{M}{\log M+4\log\log M}&\le_! \frac{M}{\log (M(\log M)^3+M)} - \frac{M(\log M)^3+M}{(\log M)^5},\\
		\frac{M}{\log M-4\log\log M}&\ge_! \frac{M}{\log (M(\log M)^{-3})} + \frac{M(\log M)^3+M}{(\log M)^5},
	\end{align*}
	the required inequalities \eqref{eq:pi-interval-1} and \eqref{eq:pi-interval-2} are established.
\end{proof}

\section{Rolling dice}

We start by pointing out that the parameter $R_1$ in \eqref{eq:para-1} is chosen as we want to truncate the expectation
\begin{align*}
	\bE(\tau^{(M)}) = \sum_{r\ge 1}r \Prob_{\sfP}(r)
\end{align*}
according to whether the number $r$ of rolls is no larger than $R_1$ or not. To be precise, we shall prove the following asymptotic relations that immediately imply Theorem \ref{th:main}.

\begin{theorem}
	As $M\to +\infty$,
	\begin{align}
		\sum_{r= 1}^{R_1} r \Prob_{\sfP}(r) &= \log M + O(\log\log M),\label{eq:Exp-R1}\\
		\sum_{r> R_1} r \Prob_{\sfP}(r) &= o(1).\label{eq:Exp-R2}
	\end{align}
\end{theorem}

\subsection{Probabilities}

In this part, we consider the probability functions
\begin{align*}
	\Prob_{\sfP}(r)&:= \big(\tfrac{1}{M}\big)^r\cdot \#\{\eta\in \cH_\sfP:r(\eta) = r\},\\
	\Prob_{\sfN}(r)&:= \big(\tfrac{1}{M}\big)^r\cdot \#\{\eta\in \cH_\sfN:r(\eta) = r\},\\
	\Prob(r,k)&:= \big(\tfrac{1}{M}\big)^r\cdot \#\{\eta\in \cH:r(\eta) = r\text{ and }|\eta|=k\}.
\end{align*}
Clearly, if $r(\eta)=r$, then $r\le |\eta|\le Mr$. It follows that $\Prob(r,k)$ is zero whenever $k$ is not in the interval $[r,Mr]$. Also,
\begin{align}
	\Prob_{\sfP}(r) &= \sum_{\substack{r\le k\le Mr\\ k\in \mathbb{P}}} \Prob(r,k),\label{eq:P-(r,k)}\\
	\Prob_{\sfN}(r) &= \sum_{\substack{r\le k\le Mr\\ k\not\in \mathbb{P}}} \Prob(r,k).\label{eq:N-(r,k)}
\end{align}
Meanwhile, for $r\ge 2$, if $(\eta_1,\ldots,\eta_r)\in \cH_\sfP$, we have $(\eta_1,\ldots,\eta_{r-1})\in \cH_\sfN$. Hence,
\begin{align*}
	\Prob_{\sfP}(r) &= \sum_{\substack{r-1\le k\le M(r-1)\\k\not\in \mathbb{P}}} \Prob(r-1,k)\cdot \frac{\#\{j\in \{1,\ldots,M\}:k+j\in \mathbb{P}\}}{M}.
\end{align*}
It follows that
\begin{align}\label{eq:P-N-lower}
	\Prob_{\sfP}(r)\ge \Prob_{\sfN}(r-1)\cdot \frac{\underset{r-1\le k\le M(r-1)}{\min}\#\{j\in \{1,\ldots,M\}:k+j\in \mathbb{P}\}}{M},
\end{align}
and
\begin{align}\label{eq:P-N-upper}
	\Prob_{\sfP}(r)\le \Prob_{\sfN}(r-1)\cdot \frac{\underset{r-1\le k\le M(r-1)}{\max}\#\{j\in \{1,\ldots,M\}:k+j\in \mathbb{P}\}}{M}.
\end{align}
Similarly,
\begin{align}\label{eq:N-N-lower}
	\Prob_{\sfN}(r)\ge \Prob_{\sfN}(r-1)\cdot \frac{\underset{r-1\le k\le M(r-1)}{\min}\#\{j\in \{1,\ldots,M\}:k+j\not\in \mathbb{P}\}}{M},
\end{align}
and
\begin{align}\label{eq:N-N-upper}
	\Prob_{\sfN}(r)\le \Prob_{\sfN}(r-1)\cdot \frac{\underset{r-1\le k\le M(r-1)}{\max}\#\{j\in \{1,\ldots,M\}:k+j\not\in \mathbb{P}\}}{M}.
\end{align}
In particular,
\begin{align}
	\Prob_{\sfP}(r)\le \Prob_{\sfN}(r-1),\label{eq:P-r-(r-1)}\\
	\Prob_{\sfN}(r)\le \Prob_{\sfN}(r-1).\label{eq:N-r-(r-1)}
\end{align}

Now we bound $\Prob_{\sfN}(r)$ for $r\le R_1$ where $R_1$ is as in \eqref{eq:para-1}.

\begin{lemma}\label{le:Pr-N-bounds}
	Suppose $M\ge M_0$. Then for $r\le R_1$,
	\begin{align}
		\Prob_{\sfN}(r)&\ge \left(1-\frac{1}{\log M-4\log\log M}\right)^{r},\label{eq:Pr-N-bounds-l}\\
		\Prob_{\sfN}(r)&\le \left(1-\frac{1}{\log M+4\log\log M}\right)^{r}.\label{eq:Pr-N-bounds-u}
	\end{align}
\end{lemma}

\begin{proof}
	When $r=1$, we have
	\begin{align*}
		\Prob_{\sfN}(1) = \frac{M-\pi(M)}{M}.
	\end{align*}
	Invoking Lemma \ref{le:pi-interval} with $k=0$ gives the inequalities. Now assuming that $2\le r\le R_1$, we have $M(r-1)\le M(R_1-1)\le M(\log M)^3$. In light of Lemma \ref{le:pi-interval},
	\begin{align*}
		\frac{\underset{r-1\le k\le M(r-1)}{\min}\#\{j\in \{1,\ldots,M\}:k+j\not\in \mathbb{P}\}}{M}&\ge 1-\frac{1}{\log M-4\log\log M},\\
		\frac{\underset{r-1\le k\le M(r-1)}{\max}\#\{j\in \{1,\ldots,M\}:k+j\not\in \mathbb{P}\}}{M}&\le 1-\frac{1}{\log M+4\log\log M}.
	\end{align*}
	Repeatedly applying \eqref{eq:N-N-lower} and \eqref{eq:N-N-upper} gives the required claims.
\end{proof}

We also have bounds for $\Prob_{\sfP}(r)$.

\begin{lemma}\label{le:Pr-P-bounds}
	Suppose $M\ge M_0$. Then for $r\le R_1$,
	\begin{align}
		\Prob_{\sfP}(r)&\ge \frac{1}{\log M+4\log\log M}\left(1-\frac{1}{\log M-4\log\log M}\right)^{r-1},\label{eq:Pr-P-bounds-l}\\
		\Prob_{\sfP}(r)&\le \frac{1}{\log M-4\log\log M}\left(1-\frac{1}{\log M+4\log\log M}\right)^{r-1}.\label{eq:Pr-P-bounds-u}
	\end{align}
\end{lemma}

\begin{proof}
	For $r=1$, we deduce the two inequalities by the fact that
	\begin{align*}
		\Prob_{\sfP}(1) = \frac{\pi(M)}{M}.
	\end{align*}
	For $2\le r\le R_1$, it still follows from Lemma \ref{le:pi-interval} that
	\begin{align*}
		\frac{\underset{r-1\le k\le M(r-1)}{\min}\#\{j\in \{1,\ldots,M\}:k+j\in \mathbb{P}\}}{M}&\ge \frac{1}{\log M+4\log\log M},\\
		\frac{\underset{r-1\le k\le M(r-1)}{\max}\#\{j\in \{1,\ldots,M\}:k+j\in \mathbb{P}\}}{M}&\le \frac{1}{\log M-4\log\log M}.
	\end{align*}
	To conclude our results, we simply combine Lemma \ref{le:Pr-N-bounds} with \eqref{eq:P-N-lower} and \eqref{eq:P-N-upper}.
\end{proof}

Our last piece of ingredient is an analog of Proposition 1 in \cite{AM2023}.

\begin{lemma}
	For every $r$ and for every non-prime $k$ with $r\le k\le Mr$,
	\begin{align}\label{eq:Pr(r,k)}
		\Prob(r,k)\le \left(1-\frac{1}{M}\right)^{\pi(k)}.
	\end{align}
\end{lemma}

\begin{proof}
	We prove this inequality by induction on $r$. When $r=1$, we observe that for every non-prime $k$ with $1\le k\le M$,
	\begin{align*}
		\Prob(r,k) = \frac{1}{M}.
	\end{align*}
	On the other hand,
	\begin{align*}
		\left(1-\frac{1}{M}\right)^{\pi(k)}\ge 1-\frac{\pi(k)}{M}.
	\end{align*}
	It is clear that for every $1\le k\le M$,
	\begin{align*}
		1-\frac{\pi(k)}{M}\ge \frac{1}{M},
	\end{align*}
	since $\pi(k)+1\le \pi(M)+1\le M$, the latter of which follows as we always have the non-prime $1$.
	
	Now let us assume that our claim holds for a certain $r-1$ (and every corresponding $k$), and we prove it for $r$. Note that by our inductive assumption,
	\begin{align*}
		\Prob(r,k) &= \sum_{\substack{k-M\le k'\le k-1\\r-1\le k'\le M(r-1)\\k'\not\in\mathbb{P}}} \frac{\Prob(r-1,k')}{M}\\
		&\le \sum_{\substack{k-M\le k'\le k-1\\r-1\le k'\le M(r-1)\\k'\not\in\mathbb{P}}}\frac{1}{M} \left(1-\frac{1}{M}\right)^{\pi(k')}\\
		&\le \sum_{\substack{k-M\le k'\le k-1\\k'\not\in\mathbb{P}}}\frac{1}{M} \left(1-\frac{1}{M}\right)^{\pi(k')}.
	\end{align*}
	Suppose that there are $m$ primes in $\{k-M,\ldots,k-1\}$, namely, $\pi(k-1)-\pi(k-M-1)=m$. We also observe that $k$ is a non-prime so that $\pi(k)=\pi(k-1)$. Hence, for every non-prime $k'$ with $k-M\le k'\le k-1$, we have $\pi(k')\ge \pi(k-M-1) = \pi(k-1)-m=\pi(k)-m$. It turns out that
	\begin{align*}
		\Prob(r,k) &\le \sum_{\substack{k-M\le k'\le k-1\\k'\not\in\mathbb{P}}} \frac{1}{M}\left(1-\frac{1}{M}\right)^{\pi(k)-m}\\
		&=(M-m)\cdot \frac{1}{M}\left(1-\frac{1}{M}\right)^{\pi(k)-m}.
	\end{align*}
	Finally, invoking the fact that
	\begin{align*}
		\frac{M-m}{M}\le \left(1-\frac{1}{M}\right)^m,
	\end{align*}
	the required result follows.
\end{proof}

\begin{definition*}
	In the next two subsections, we shall use the following notation to facilitate expression:
	\begin{align*}
		L=L(M)&:=\log M-4\log\log M,\\
		U=U(M)&:=\log M+4\log\log M.
	\end{align*}
\end{definition*}

\subsection{Less than $(\log M)^3$ rolls? There we are!}

In this part, we shall establish \eqref{eq:Exp-R1}, namely, the main contribution to the expected stopping round. The basic idea is to bound the left-hand side of \eqref{eq:Exp-R1} from both above and below:
\begin{align}
	\sum_{r= 1}^{R_1} r \Prob_{\sfP}(r) &\le \log M + O(\log\log M),\label{eq:Exp-R1-u}\\
	\sum_{r= 1}^{R_1} r \Prob_{\sfP}(r) &\ge \log M + O(\log\log M).\label{eq:Exp-R1-l}
\end{align}

\begin{proof}[Proof of \eqref{eq:Exp-R1-u}]
	With a direct application of \eqref{eq:Pr-P-bounds-u}, we have
	\begin{align*}
		\sum_{r= 1}^{R_1} r \Prob_{\sfP}(r) &\le \frac{1}{L} \sum_{r= 1}^{R_1} r \left(1-\frac{1}{U}\right)^{r-1}\\
		&\le \frac{1}{L} \sum_{r\ge 1} r \left(1-\frac{1}{U}\right)^{r-1}\\
		&=\frac{(\log M+4\log\log M)^2}{\log M-4\log\log M}\\
		&=\log M + O(\log\log M),
	\end{align*}
	as requested.
\end{proof}

\begin{proof}[Proof of \eqref{eq:Exp-R1-l}]
	We make use of \eqref{eq:Pr-P-bounds-l} and derive that
	\begin{align*}
		\sum_{r= 1}^{R_1} r \Prob_{\sfP}(r) &\ge \frac{1}{U} \sum_{r= 1}^{R_1} r \left(1-\frac{1}{L}\right)^{r-1}\\
		&= \frac{1}{U} \sum_{r\ge 1} r \left(1-\frac{1}{L}\right)^{r-1} - \frac{1}{U} \sum_{r\ge R_1+1} r \left(1-\frac{1}{L}\right)^{r-1}.
	\end{align*}
	For the former term, it is straightforward that
	\begin{align*}
		\frac{1}{U} \sum_{r\ge 1} r \left(1-\frac{1}{L}\right)^{r-1} &=\frac{(\log M-4\log\log M)^2}{\log M+4\log\log M}\\
		&=\log M + O(\log\log M). 
	\end{align*}
	For the latter term, we have
	\begin{align*}
		\frac{1}{U} \sum_{r\ge R_1+1} r \left(1-\frac{1}{L}\right)^{r-1} = \frac{L(L+R_1)}{U}\left(1-\frac{1}{L}\right)^{R_1}\xrightarrow{M\to +\infty} 0,
	\end{align*}
	so it only gives a marginal contribution of asymptotic order $o(1)$.
\end{proof}

\subsection{More than $(\log M)^3$ rolls? Bazinga!}

Now we prove the very marginal error term for $\bE(\tau^{(M)})$ as claimed in \eqref{eq:Exp-R2}. Our strategy is a two-step process. That is, we aim to establish two separate estimates, where $R_2$ is as in \eqref{eq:para-1}:
\begin{align}
	\sum_{r= R_1+1}^{R_2} r \Prob_{\sfP}(r) &= o(1),\label{eq:Exp-R21}\\
	\sum_{r\ge R_2+1} r \Prob_{\sfP}(r) &= o(1).\label{eq:Exp-R22}
\end{align}

\begin{proof}[Proof of \eqref{eq:Exp-R21}]
	Recalling \eqref{eq:P-r-(r-1)}, \eqref{eq:N-r-(r-1)} and \eqref{eq:Pr-N-bounds-u}, we find that for $R_1+1\le r\le R_2$,
	\begin{align*}
		\Prob_{\sfP}(r)\le \Prob_{\sfN}(R_1)\le \left(1-\frac{1}{U}\right)^{R_1}.
	\end{align*}
	Therefore,
	\begin{align*}
		\sum_{r= R_1+1}^{R_2} r \Prob_{\sfP}(r)&\le \left(1-\frac{1}{U}\right)^{R_1}\sum_{r= R_1+1}^{R_2} r\\
		&\le \left(1-\frac{1}{U}\right)^{R_1}\cdot R_2^2\\
		&\le \left(1-\frac{1}{\log M+4\log\log M}\right)^{(\log M)^3 - 1}\cdot M^4\\
		&\to 0,
	\end{align*}
	as $M\to +\infty$.
\end{proof}

\begin{proof}[Proof of \eqref{eq:Exp-R22}]
	We start by using \eqref{eq:P-r-(r-1)} to bound
	\begin{align*}
		\sum_{r\ge R_2+1} r \Prob_{\sfP}(r)&\le \sum_{r\ge R_2+1} r \Prob_{\sfN}(r-1)\\
		&= \sum_{r\ge R_2} (r+1) \Prob_{\sfN}(r).
	\end{align*}
	Invoking \eqref{eq:N-(r,k)} then gives
	\begin{align*}
		\sum_{r\ge R_2+1} r \Prob_{\sfP}(r)&\le \sum_{r\ge R_2} (r+1) \sum_{\substack{r\le k\le Mr\\ k\not\in \mathbb{P}}} \Prob(r,k)\\
		&=\sum_{\substack{k\ge R_2\\k\not\in\mathbb{P}}}\sum_{r=\max\{R_2,\lceil\frac{k}{M}\rceil\}}^k (r+1)\Prob(r,k).
	\end{align*}
	It further follows from \eqref{eq:Pr(r,k)} that
	\begin{align*}
		\sum_{r\ge R_2+1} r \Prob_{\sfP}(r)&\le \sum_{\substack{k\ge R_2\\k\not\in\mathbb{P}}}\left(1-\frac{1}{M}\right)^{\pi(k)}\sum_{r=\max\{R_2,\lceil\frac{k}{M}\rceil\}}^k (r+1)\\
		&\le \sum_{k\ge R_2} (k+1)^2 \left(1-\frac{1}{M}\right)^{\pi(k)}.
	\end{align*}
	Now for each integer $\ell\ge \pi(R_2)$, we define
	\begin{align*}
		S_\ell := \{k\in\mathbb{N}: \pi(k)=\ell\}.
	\end{align*}
	Clearly, $(\min S_\ell)> \ell$ since it is always true that $\pi(\ell)< \ell$. We also claim that $(\max S_\ell)< \ell^2$.
	To see this, we recall our assumption that $M\ge M_0$. It is then clear that $\ell^2\ge \pi(R_2)^2 \ge \frac{M_0}{(\log M_0)^3}$, so that we may use \eqref{eq:pi-lower} for $x=\ell^2$ and obtain
	\begin{align*}
		\pi(\ell^2)\ge \frac{\ell^2}{2\log \ell}> \ell+1.
	\end{align*}
	The latter inequality is true for all $\ell\ge 3$, so it can be coordinated by a suitable choice of $M_0$. The above arguments indicate that $S_\ell$ is contained in the interval $[\ell+1,\ell^2-1]$. Thus,
	\begin{align*}
		\sum_{r\ge R_2+1} r \Prob_{\sfP}(r)&\le \sum_{\ell\ge \pi(R_2)}\left(1-\frac{1}{M}\right)^{\ell}\sum_{k\in S_\ell} (k+1)^2\\
		&\le \sum_{\ell\ge \pi(R_2)}\left(1-\frac{1}{M}\right)^{\ell}\sum_{k=\ell+1}^{\ell^2-1} (k+1)^2\\
		&\le \sum_{\ell\ge \pi(R_2)}\ell^8 \left(1-\frac{1}{M}\right)^{\ell}.
	\end{align*}
	Finally, we recall that $\pi(R_2)=\pi(M^2)\ge \frac{M^2}{2\log M}$ by invoking \eqref{eq:pi-lower}. Now a trivial computation gives
	\begin{align*}
		\sum_{\ell\ge \frac{M^2}{2\log M}}\ell^8 \left(1-\frac{1}{M}\right)^{\ell}\xrightarrow{M\to +\infty} 0,
	\end{align*}
	thereby confirming the claimed estimate.
\end{proof}

\section{Conclusion}

There are a handful of problems that merit further investigation. First, from the limited data in Table \ref{ta:value}, it seems that $\bE(\tau^{(M)})$ is always larger than $\log M$. If this hunch is true, then a highly feasible approach is to study the asymptotic behavior of $\bE(\tau^{(M)}) - \log M$, unless it behaves fluctuantly.

\begin{problem}
	Is it true that $\bE(\tau^{(M)}) > \log M$ for all $M\ge 2$? If not, does the sign of $\bE(\tau^{(M)}) - \log M$ change infinitely many times?
\end{problem}

\begin{problem}
	Is it possible to elaborate on the error term in \eqref{eq:main}?
\end{problem}

In addition, Martinez and Zeilberger \cite{MZ2023} considered the scenarios where other kinds of numbers are hit. This problem is not hard when these numbers distribute evenly and have a positive density $\delta>0$ in the set of positive integers. (\textit{Answer}: $\approx \frac{1}{\delta}$ rolls of fair dice with a sufficiently large number of faces.) But unfortunately, for the \emph{hitting-a-square} case, our method cannot be transplanted, and thus a corresponding asymptotic relation is out of reach. This is mainly because primes are much \emph{denser} than squares among positive integers, even though both have density zero.

\begin{problem}
	If we replace primes with any kind of positive integers as our hitting targets, do we still have a corresponding asymptotic relation?
\end{problem}

\subsection*{Acknowledgements}

The author was supported by a Killam Postdoctoral Fellowship from the Killam Trusts.

\bibliographystyle{amsplain}

\end{document}